\documentclass[12pt,a4paper]{amsart}

\def\rife#1{(\ref{#1})}
\newcommand{\disp}{\displaystyle}
\newtheorem{theo}{\sc Theorem}
\newtheorem{lemma}[theo]{\sc Lemma}
\newtheorem{ohss}[theo]{\sc Remark}
\def\R{\mathbb{R}}
\def\N{\mathbb{N}}
\def\be{\begin{equation}}
\def\ee{\end{equation}}
\newcommand{\arrstre}{\renewcommand{\arraystretch}{2}}
\def\un{u_{n}}
\def\wm{w_{M}}
\def\fn{f_{n}}
\def\elle#1{L^{#1}(\Omega)}
\def\wq{W_{0}^{1,q}(\Omega)}
\def\huz{H_{0}^{1}(\Omega)}
\def\ww{W_{0}^{1,1}(\Omega)}
\def\io{\int_{\Omega}}
\def\norma#1#2{\|#1\|_{\lower 4pt \hbox{$\scriptstyle #2$}}}

\begin{document}

\title{$W^{1,1}_0$ minima of non coercive functionals}

\author{Lucio Boccardo, Gisella Croce, Luigi Orsina}
\address{L.B. -- Dipartimento di Matematica, ``Sapienza'' Universit\`{a} di Roma,
P.le A. Moro 2, 00185 Roma (ITALY)}
\email{boccardo@mat.uniroma1.it}
\address{G.C. -- Laboratoire de Math\'ematiques Appliqu\'ees du Havre, Universit\'e du Havre,
25, rue Philippe Lebon, 76063 Le Havre (FRANCE)}
\email{gisella.croce@univ-lehavre.fr}
\address{L.O. -- Dipartimento di Matematica, ``Sapienza'' Universit\`{a} di Roma,
P.le A. Moro 2, 00185 Roma (ITALY)}
\email{orsina@mat.uniroma1.it}

\begin{abstract}
We study an integral non coercive functional defined on $\huz$, proving the existence of a minimum in $\ww$.
\end{abstract}

\maketitle


In this paper we study a class of integral functionals defined on $\huz$, but non coercive on the same space, so that the standard approach of the Calculus of Variations does not work. However, the functionals are coercive on
$\ww$ and we will prove the existence of minima, despite the non reflexivity of $\ww$, which implies that, in general, the Direct Methods fail due to lack of compactness.

\smallskip

Let $J$ be the functional defined as
$$
J(v) = \io\frac{j(x,\nabla v)}{[1+b(x)|v|]^{2}} + \frac12\io|v|^2 - \io f\,v\,,
\quad
v \in \huz\,.
$$
We assume that $\Omega$ is a bounded open set of $\R^{N}$, $N > 2$, that
$j : \Omega \times \R^{N} \to \R$ is such that $j(\cdot,\xi)$ is measurable on $\Omega$ for every $\xi$ in $\R^{N}$, $j(x,\cdot)$ is convex and belongs to $C^{1}(\R^{N})$ for almost every $x$ in $\Omega$, and
\be\label{hypj}
\alpha |\xi|^{2} \leq j(x,\xi) \leq \beta |\xi|^{2}\,,
\ee
\be\label{hypjp}
|j_{\xi}(x,\xi)| \leq \gamma |\xi|\,,
\ee
for some positive $\alpha$, $\beta$ and $\gamma$, for almost every $x$ in $\Omega$, and for every $\xi$ in $\R^{N}$. We assume that $b$ is a measurable function on $\Omega$ such that
\be\label{hypb}
0 \leq b(x) \leq B\,,
\quad
\mbox{for almost every $x$ in $\Omega$,}
\ee
where $B > 0$, while $f$ belongs to some Lebesgue space. For $k > 0$ and $s \in \R$, we define the truncature function as $T_{k}(s) = \max(-k,\min(s,k))$.

In \cite{BO-Pisa} the minimization in $\huz$ of the functional
$$
I(v) = \io\frac{j(x,\nabla v)}{[1+|v|]^{\theta}} - \io f\,v\,,
\quad 0<\theta<1,
\;
f\in\elle m\,,
$$
was studied.
It was proved that $I(v)$ is coercive on the Sobolev space $\wq$, for some $q = q(\theta,m)$ in $(1,2)$, and that $I(v)$ achieves its minimum on $\wq$. This approach does not work for $\theta > 1$ (see Remark \ref{coerc} below). Here we will able to overcome this difficulty thanks to the presence of the lower order term $\io|v|^2$, which will yield the coercivity of $J$ on $\ww$; then we will prove the existence of minima in $\ww$, even if it is a non reflexive space.

Integral functionals like $J$ or $I$ are studied in \cite{bbl}, in the context of the Thomas-Fermi-von Weizs\"acker theory.

We are going to prove the following result.
\begin{theo}\label{main_theorem}\sl
Let $f \in L^2(\Omega)$. Then there exists $u$ in $\ww\cap\elle2$
minimum of $J$, that is,
\be\label{enunciato}
\io \!\frac{j(x,\nabla u)} {[1+b(x)|u|]^{2}}+\frac12\io\!|u|^2
-\io f u
 \leq
\io \!\frac{j(x,\nabla v)} {[1+b(x)|v|]^{2}}+\frac12\io\!|v|^2
-\io f v,
\ee
for every $v$ in $\huz$. Moreover
 $T_k(u)$ belongs to $\huz$ for every $k>0$.
\end{theo}
In \cite{BCO}
we studied the following elliptic boundary problem:
\be\left\{
\arrstre
\begin{array}{cl}
\disp
-{\rm div}\bigg(\frac{a(x)\,\nabla u}{(1+b(x)|u|)^{2}}\bigg) + u = f & \mbox{in $\Omega $,}\\
\hfill u = 0 \hfill & \mbox{on $\partial\Omega $,}
\end{array}
\right.
\label{lineare}
\ee
under the same assumptions on $\Omega$, $b$ and $f$,
with
$0 <\alpha\leq a(x) \leq\beta$.
It is easy to see that the Euler equation of $J$, with $j(x,\xi)=\frac12 a(x)|\xi|^2$, is not
equation (\ref{lineare}). Therefore Theorem \ref{main_theorem} cannot be deduced from \cite{BCO}.
Nevertheless some technical steps of the two papers (for example, the a priori estimates) are similar.

\smallskip

We will prove Theorem \ref{main_theorem} by approximation. Therefore,
we begin with the case of bounded data.

\begin{lemma}\label{esistelinfty}\sl
If $g$ belongs to $\elle\infty$, then there exists a minimum $w$ belonging to $\huz \cap \elle\infty$ of the functional
$$
v\in \huz \mapsto \io\frac{j(x,\nabla v)}{[1+b(x)|v|]^{2}} + \frac12\io|v|^2 - \io g\,v\,.
$$
\end{lemma}

\begin{proof} Since the functional is not coercive on $\huz$, we cannot directly apply the standard techniques of the Calculus of Variations. Therefore, we begin by approximating it. Let $M > 0$, and let $J_{M}$ be the functional defined as
$$
J_{M}(v) = \io\frac{j(x,\nabla v)}{[1+b(x)|T_{M}(v)|]^{2}} + \frac12\io|v|^2 - \io g\,v\,,
\quad
v \in \huz\,.
$$
Since $J_{M}$ is both weakly lower semicontinuous (due to the convexity of $j$ and to De Giorgi's theorem, see \cite{DG}) and coercive on $\huz$, for every $M > 0$ there exists a minimum $\wm$ of $J_{M}$ on $\huz$. Let $A = \norma{g}{\elle\infty}$, let $M > A$, and consider the inequality $J_{M}(\wm) \leq J_{M}(T_{A}(\wm))$, which holds true since $\wm$ is a minimum of $J_{M}$. We have
$$
\arrstre
\begin{array}{l}
\disp
\io \frac{j(x,\nabla\wm)}{[1 + b(x)|T_{M}(\wm)|]^{2}} + \frac12 \io |\wm|^{2} - \io g\,\wm
\\
\disp
\quad
\leq
\io \frac{j(x,\nabla T_{A}(\wm))}{[1 + b(x)|T_{M}(T_{A}(\wm))|]^{2}} + \frac12 \io |T_{A}(\wm)|^{2} - \io g\,T_{A}(\wm)
\\
\disp
\quad
=
\int_{\{|\wm|\leq A\}}\!\! \frac{j(x,\nabla \wm)}{[1 + b(x)|T_{M}(\wm)|]^{2}} + \frac12 \io |T_{A}(\wm)|^{2} - \io g T_{A}(\wm)
\,,
\end{array}
$$
where, in the last passage, we have used that $T_{M}(T_{A}(\wm)) = T_{M}(\wm)$ on the set $\{|\wm|\leq A\}$, and that $j(x,0) = 0$. Simplifying equal terms, we thus get
$$
\arrstre
\begin{array}{l}
\disp
\int_{\{|\wm| \geq M\}} \frac{j(x,\nabla\wm)}{[1 + b(x)|T_{M}(\wm)|]^{2}}
\\
\disp
\quad
+
\frac12 \io [|\wm|^{2} - |T_{A}(\wm)|^{2}]
\leq
\io g\,[\wm - T_{A}(\wm)]\,.
\end{array}
$$
Dropping the first term, which is nonnegative, we obtain
$$
\frac12 \io [\wm - T_{A}(\wm)]\,[\wm + T_{A}(\wm)]
\leq
\io g\,[\wm - T_{A}(\wm)]\,,
$$
which can be rewritten as
$$
\frac12 \io [\wm - T_{A}(\wm)]\,[\wm + T_{A}(\wm) - 2g] \leq 0\,.
$$
We then have, since $\wm = T_{A}(\wm)$ on the set $\{|\wm| \leq A\}$,
$$
\frac12 \int_{\{\wm > A\}} \! [\wm - A][\wm + A - 2g]
+
\frac12 \int_{\{\wm < -A\}} \! [\wm + A][\wm - A - 2g]
\leq
0\,.
$$
Since $|g| \leq A$, we have $A - 2g \geq -A$, and $-A-2g < A$, so that
$$
0 \leq \frac12 \int_{\{\wm > A\}} [\wm - A]^{2} + \frac12 \int_{\{\wm < -A\}}[\wm + A]^{2} \leq 0\,,
$$
which then implies that ${\rm meas}(\{|\wm| \geq A\}) = 0$, and so $|\wm| \leq A$ almost everywhere in $\Omega$. Recalling the definition of $A$, we thus have
\be\label{stimalinfty}
\norma{\wm}{\elle\infty} \leq \norma{g}{\elle\infty}\,.
\ee
Since $M > \norma{g}{\elle\infty}$, we thus have $T_{M}(\wm) = \wm$. Starting now from $J_{M}(\wm) \leq J_{M}(0) = 0$ we obtain, by \rife{stimalinfty},
$$
\io \frac{j(x,\nabla\wm)}{[1 + b(x)|\wm|]^{2}} + \frac12 \io |\wm|^{2} \leq \io g\,\wm \leq {\rm meas}(\Omega)\,\norma{g}{\elle\infty}^{2}\,,
$$
which then implies, by \rife{hypj} and \rife{hypb}, and dropping the nonnegative second term,
$$
\frac{\alpha} {[1+B\norma{g}{\elle\infty}]^{2}}\io|\nabla w_M|^2
\leq
{\rm meas}(\Omega)\,\norma{g}{\elle\infty}^{2}\,.
$$
Thus, $\{\wm\}$ is bounded in $\huz \cap \elle\infty$, and so, up to subsequences, it converges to some function $w$ in $\huz \cap \elle\infty$ weakly in $\huz$, strongly in $\elle2$, and almost everywhere in $\Omega$. We prove now that
\be\label{wlsc}
\io \frac{j(x,\nabla w)}{[1 + b(x)|w|]^{2}}
\leq
\liminf_{M \to +\infty}\,\io \frac{j(x,\nabla\wm)}{[1 + b(x)|\wm|]^{2}}\,.
\ee
Indeed, since $j$ is convex, we have
$$
\arrstre
\begin{array}{l}
\disp
\io \frac{j(x,\nabla\wm)}{[1 + b(x)|\wm|]^{2}}
\\
\disp
\quad
\geq
\io \frac{j(x,\nabla w)}{[1 + b(x)|\wm|]^{2}}
-
\io \frac{j_{\xi}(x,\nabla w)}{[1 + b(x)|\wm|]^{2}} \cdot \nabla [\wm - w]\,.
\end{array}
$$
Using assumption \rife{hypj}, the fact that $w$ belongs to $\huz$, the almost everywhere convergence of $\wm$ to $w$ and Lebesgue's theorem, we have
\be\label{wlsc1}
\lim_{M \to +\infty}\,\io \frac{j(x,\nabla w)}{[1 + b(x)|\wm|]^{2}} = \io \frac{j(x,\nabla w)}{[1 + b(x)|w|]^{2}}\,.
\ee
Using assumption \rife{hypjp}, the fact that $w$ belongs to $\huz$, and the almost everywhere convergence of $\wm$ to $w$, we have by Lebesgue's theorem that
$$
\lim_{M \to +\infty}\,\frac{j_{\xi}(x,\nabla w)}{[1 + b(x)|\wm|]^{2}} = \frac{j_{\xi}(x,\nabla w)}{[1 + b(x)|w|]^{2}}\,,
\quad
\mbox{strongly in $(\elle2)^{N}$.}
$$
Since $\nabla\wm$ tends to $\nabla w$ weakly in the same space, we thus have that
\be\label{wlsc2}
\lim_{M \to +\infty}\,\io \frac{j_{\xi}(x,\nabla w)}{[1 + b(x)|\wm|]^{2}} \cdot \nabla [\wm - w] = 0\,.
\ee
Using \rife{wlsc1} and \rife{wlsc2}, we have that \rife{wlsc} holds true. On the other hand, using \rife{hypj} and Lebesgue's theorem again, it is easy to see that
$$
\lim_{M \to +\infty}\,\io \frac{j(x,\nabla v)}{[1 + b(x)|T_{M}(v)|]^{2}} = \io \frac{j(x,\nabla v)}{[1 + b(x)|v|]^{2}}\,,
\quad
\forall v \in \huz\,.
$$
Thus, starting from $J_{M}(\wm) \leq J_{M}(v)$, we can pass to the limit as $M$ tends to infinity (using also the strong convergence of $\wm$ to $w$ in $\elle2$), to have that $w$ is a minimum.
\end{proof}

As stated before,
we prove Theorem \ref{main_theorem} by approximation. More in detail, if $f_n=T_n(f)$ then
Lemma \ref{esistelinfty} with $g = \fn$ implies that there exists a minimum $\un$ in $\huz \cap \elle\infty$ of the functional
$$
J_n(v)=
\io\frac{j(x,\nabla v)} {[1+b(x)|v|]^{2}}+\frac12\io|v|^2-\io f_n\, v\,,\quad v \in \huz\,.
$$
In the following lemma we prove some uniform estimates on $u_n$.

\begin{lemma}\label{tuttestime}\sl
Let $u_n$ in $\huz\cap \elle{\infty}$ be a minimum of $J_n$.
Then
\be\label{primastima}
\io \frac{|\nabla \un|^{2}}{(1 + b(x)|\un|)^{2}} 
\leq
\frac{1}{2\alpha}\io|f|^2\,;
\ee
\be\label{tk}
\io|\nabla T_k(u_n)|^{2}
\leq
\frac{{(1+B\,k)^{2}}}{2\alpha}\io|f|^2\,;
\ee
\be\label{secondastima}
\io|\un|^2 
\leq
4\io|f|^2 \,;
\ee
\be\label{terzastima}
\io|\nabla\un|
\leq
\bigg[\frac{1}{2\alpha} \io|f|^2\bigg]^\frac12
\bigg({\rm meas}(\Omega)^\frac12 + 2B\bigg[\io|f|^2\bigg]^\frac12\bigg)\,;
\ee
\be\label{gk}
\io| G_k(\un)|^2
\leq
4\int_{\{|\un| \geq k\}}|f|^2\,,
\ee
where $G_{k}(s) = s - T_{k}(s)$ for $k \geq 0$ and $s$ in $\R$.
\end{lemma}

\begin{proof}
The minimality of $\un$ implies that $J_n(\un) \leq J_n(0)$, that is,
\be\label{0injn}
\io\frac{j(x,\nabla\un)} {[1+b(x)|\un|]^{2}}
+\frac12\io u_n^2\leq \io f_n\,u_n\,.
\ee
Using \rife{hypj} on the left hand side, and Young's inequality on the right hand side gives
$$
\alpha\io\frac{|\nabla\un|^{2}} {[1+b(x)|\un|]^{2}}
+\frac12\io u_n^2\leq \frac12\io u_n^2+\frac12\io f_n^2\,,
$$
which then implies (\ref{primastima}).
Let now $k\geq 0$. The above estimate, and \rife{hypb}, give
$$
\frac{1}{(1+Bk)^{2}} \io {|\nabla T_k(u_n)|^{2}}
\leq
\int_{\{|\un|\leq k\}} \frac{|\nabla u_n|^{2}}{[1+b(x)|u_n|]^{2}}
\leq
\frac1{2\alpha}\io|f|^2\,,
$$
and therefore (\ref{tk}) is proved.
On the other hand, dropping the first positive term in \rife{0injn} and using H\"older's inequality on the right hand side, we have
$$
\frac12\io|\un|^2
\leq\io|f_n\un|
\leq
\bigg[\io|\fn|^2\bigg]^\frac12
\bigg[\io|\un|^2\bigg]^\frac12,
$$
that is, (\ref{secondastima}) holds.
H\"older's inequality, assumption \rife{hypb}, and estimates (\ref{primastima}) and (\ref{secondastima}) give (\ref{terzastima}):
\be\label{coercitivo}
\arrstre
\begin{array}{r@{\hspace{2pt}}c@{\hspace{2pt}}l}
\disp
\io|\nabla\un|
& \leq &
\disp
\bigg[ \io\frac{|\nabla\un|^2} {[1+b(x)|\un|]^{2}}\bigg]^\frac12
\bigg[\io{[1+b(x)|\un|]^{2}}\bigg]^\frac12
\\
& \leq &
\disp
\bigg[\frac{1}{2\alpha} \io|f|^2\bigg]^\frac12
\bigg({\rm meas}(\Omega)^\frac12 + 2B\bigg[\io|f|^2\bigg]^\frac12\bigg)\,.
\end{array}
\ee
We are left with estimate (\ref{gk}).
Since $J_n(u_n)\leq J_n(T_k(u_n))$ we have
$$
\arrstre
\begin{array}{l}
 \disp
\frac12\io\frac{j(x,\nabla\un)} {[1+b(x)|\un|]^{2}}+\frac12\io|\un|^2
 {-\io f_n\un }
\\
\disp
\leq
\frac12\io\frac{j(x,\nabla T_k(\un))} {[1+b(x)|T_k(\un)|]^{2}}+\frac12\io|T_k(\un)|^2-\io f_n T_k(\un)\,.
\end{array}
$$
Recalling the definition of $G_k(s)$, and using that $|s|^{2} - |T_{k}(s)|^{2} \geq |G_{k}(s)|^{2}$, the last inequality implies
$$
\arrstre
\begin{array}{l}
 \disp
\frac12\io\frac{j(x,\nabla G_k(\un))} {[1+b(x)|\un|]^{2}}+\frac12\io|G_k(\un)|^2
\leq
 \io f_n G_k(\un)\,.
\end{array}
$$
Dropping the first term of the left hand side and using H\"older's inequality on the right one, we obtain
$$
\frac12\io|G_k(\un)|^2
\leq\bigg[\int_{\{|\un| \geq k\}}|f|^2\bigg]^\frac12
\bigg[\io| G_k(\un)|^2\bigg]^\frac12\,,
$$
that is, (\ref{gk}) holds.
\end{proof}

\begin{lemma}\label{lemma_convergenze}\sl
Let $u_n$ in $\huz \cap \elle\infty$ be a minimum of $J_n$.
Then there exists a subsequence, still denoted by $\{\un\}$, and a function $u$ in $\ww \cap \elle2$, with $T_k(u)$ in $\huz$ for every $k>0$, such that $\un$ converges to $u$ almost everywhere in $\Omega$, strongly in $\elle2$ and weakly in $\ww$, and $T_k(\un)$ converges to $T_k(u)$ weakly in $\huz$. Moreover,
\be\label{weakconv}
\lim_{n \to +\infty}\,\frac{\nabla \un}{1 + b(x)|\un|} = \frac{\nabla u}{1 + b(x)|u|} \quad \mbox{ weakly in $(\elle2)^{N}$.}
\ee
\end{lemma}
\begin{proof}
By (\ref{terzastima}), the sequence $\un$ is bounded in $\ww$. Therefore, it is relatively compact in $\elle{1}$. Hence, up to subsequences still denoted by $\un$, there exists $u$ in $\elle{1}$ such that $\un$ almost everywhere converges to $u$. From Fatou's lemma applied to (\ref{secondastima}) we then deduce that $u$ belongs to $\elle{2}$.

We are going to prove that $u_n$ strongly converges to $u$ in $\elle2$. Let $E$ be a measurable subset of $\Omega$; then by (\ref{gk}) we have
$$
\arrstre
\begin{array}{r@{\hspace{2pt}}c@{\hspace{2pt}}l}
\disp
\int_{E}|\un|^2
&\leq &
\disp
2\int_{E }|T_k(\un)|^2+
2\int_{E }|G_k(\un)|^2
\\
\disp
&\leq &
\disp
2 k^2{\rm meas}(E)+2 \io|G_k(\un)|^2
\\
&
\leq
&
\disp 2 k^2 {\rm meas}(E)
+8 \int_{\{|\un| \geq k\}}|f|^2\,.
\end{array}
$$
Since $u_n$ is bounded in $L^2(\Omega)$ by (\ref{secondastima}), we can choose $k$ large enough so that the second integral is small, uniformly with respect to $n$; once $k$ is chosen, we can choose the measure of $E$ small enough such that the first term is small. Thus, the sequence $\{\un^{2}\}$ is equiintegrable and so, by Vitali's theorem, $\un$ strongly converges to $u$ in $\elle2$.

Now we to prove that $u_n$ weakly converges to $u$ in $\ww$.
Let $E$ be a measurable subset of $\Omega$.
By H\"older's inequality, assumption \rife{hypb}, and (\ref{primastima}), one has, for $i \in \{1,\dots,N\}$,
$$
\arrstre
\begin{array}{r@{\hspace{2pt}}c@{\hspace{2pt}}l}
\disp
\int_{E}\left|\frac{\partial\un}{\partial x_{i}}\right|
\leq
\int_{E}|\nabla\un|
& \leq &
\disp
\bigg[ \int_{E}\frac{|\nabla\un|^2} {[1+b(x)|\un|]^{2}}\bigg]^\frac12
\bigg[\int_{E}{[1+b(x)|\un|]^{2}}\bigg]^\frac12
\\
& \leq &
\disp
\bigg[\frac{1}{2\alpha} \io|f|^2\bigg]^\frac12
\bigg[\int_{E}{[1+B|\un|]^{2}}\bigg]^\frac12\,.
\end{array}
$$
Since the sequence $\{\un\}$ is compact in $\elle{2}$, this estimate implies that the sequence $\{\frac{\partial\un}{\partial x_{i}}\}$ is equiintegrable. Thus, by Dunford-Pettis theorem, and up to subsequences, there exists $Y_{i}$ in $\elle1$ such that $\frac{\partial\un}{\partial x_{i}}$ weakly converges to $Y_{i}$ in $\elle1$. Since $\frac{\partial\un}{\partial x_{i}}$ is the distributional partial derivative of $\un$, we have, for every $n$ in $\N$,
$$
\io \frac{\partial\un}{\partial x_{i}}\,\varphi = -\io \un\,\frac{\partial\varphi}{\partial x_{i}}\,,
\quad
\forall \varphi \in C^{\infty}_{0}(\Omega)\,.
$$
We now pass to the limit in the above identities, using that $\partial_{i}\un$ weakly converges to $Y_{i}$ in $\elle1$, and that $\un$ strongly converges to $u$ in $\elle2$: we obtain
$$
\io Y_{i}\,\varphi = -\io u\,\frac{\partial\varphi}{\partial x_{i}}\,,
\quad
\forall \varphi \in C^{\infty}_{0}(\Omega)\,.
$$
This implies that $Y_{i} = \frac{\partial u}{\partial x_{i}}$, and this result is true for every $i$. Since $Y_{i}$ belongs to $\elle1$ for every $i$, $u$ belongs to $\ww$, as desired.

Since by \rife{tk} it follows that the sequence $\{T_{k}(\un)\}$ is bounded in $\huz$, and since $\un$ tends to $u$ almost everywhere in $\Omega$, then $T_{k}(\un)$ weakly converges to $T_{k}(u)$ in $\huz$, and $T_k(u)$ belongs to $\huz$ for every $k \geq 0$.

Finally, we prove \rife{weakconv}. Let $\Phi$ be a fixed function in $(\elle\infty)^{N}$. Since $\un$ almost everywhere converges to $u$ in $\Omega$, we have
$$
\lim_{n \to +\infty}\,\frac{\Phi}{1 + b(x)|\un|} = \frac{\Phi}{1 + b(x)|u|}
\quad
\mbox{almost everywhere in $\Omega$.}
$$
By Egorov's theorem, the convergence is therefore quasi uniform; i.e., for every $\delta > 0$ there exists a subset $E_{\delta}$ of $\Omega$, with ${\rm meas}(E_{\delta}) < \delta$, such that
\be\label{qu}
\lim_{n \to +\infty}\,\frac{\Phi}{1 + b(x)|\un|} = \frac{\Phi}{1 + b(x)|u|}
\quad
\mbox{uniformly in $\Omega \setminus E_{\delta}$.}
\ee
We now have
$$
\arrstre
\begin{array}{l}
\disp
\left| \io \frac{\nabla \un}{1 + b(x)|\un|}\cdot \Phi - \io \frac{\nabla u}{1 + b(x)|u|} \cdot \Phi \right|
\\
\disp
\quad
\leq
\left| \int_{\Omega \setminus E_{\delta}} \nabla \un \cdot \frac{\Phi}{1 + b(x)|\un|} - \int_{\Omega \setminus E_{\delta}} \nabla u \cdot \frac{\Phi}{1 + b(x)|u|} \right|
\\
\disp
\qquad
+
\norma{\Phi}{\elle\infty} \int_{E_{\delta}}[|\nabla\un| + |\nabla u|]\,.
\end{array}
$$
Using the equiintegrability of $|\nabla\un|$ proved above, and the fact that $|\nabla u|$ belongs to $\elle1$, we can choose $\delta$ such that the second term of the right hand side is arbitrarily small, uniformly with respect to $n$, and then use \rife{qu} to choose $n$ large enough so that the first term is arbitrarily small. Hence, we have proved that
\be\label{wl1}
\lim_{n \to +\infty}\,\frac{\nabla \un}{1 + b(x)|\un|} = \frac{\nabla u}{1 + b(x)|u|}
\quad
\mbox{weakly in $(\elle1)^{N}$.}
\ee
On the other hand, from \rife{primastima} it follows that the sequence $\frac{\nabla \un}{1 + b(x)|\un|}$ is bounded in $(\elle2)^{N}$, so that it weakly converges to some function $\sigma$ in the same space. Since \rife{wl1} holds, we have that $\sigma = \frac{\nabla u}{1 + b(x)|u|}$, and \rife{weakconv} is proved.
\end{proof}

\begin{ohss}\rm
The fact that we need to prove \rife{weakconv} is one of the main differences with the paper \cite{BCO}.
\end{ohss}

\begin{proof}[Proof of Theorem \ref{main_theorem}]
Let $u_n$ be as in Lemma \ref{lemma_convergenze}.
The minimality of $\un$ implies that
\be\label{last}
\arrstre
\begin{array}{l}
\disp \io\frac{j(x,\nabla\un)} {[1+b(x)|\un|]^{2}}+\frac12\io|\un|^2
-\io f_n\un
\\
\disp
\qquad
\leq
\io\frac{j(x,\nabla v)} {[1+b(x)|v|]^{2}}+\frac12\io|v|^2
-\io f_n v
\end{array}
\ee
for every $v$ in $\huz$.
The result will then follow by passing to the limit in the previous inequality.
The right hand side of (\ref{last}) is easy to handle since $\fn$ converges to $f$ in $\elle2$.
Let us study the limit of the left hand side of (\ref{last}).
The convexity of $j$ implies that
$$
\arrstre
\begin{array}{l}
\disp
\io \frac{j(x,\nabla\un)}{[1 + b(x)|\un|]^{2}}
\geq
\io \frac{j(x,\nabla T_k(u))}{[1 + b(x)|\un|]^{2}}
\\
\disp
\quad
- \io \frac{j_{\xi}(x,\nabla T_k(u))}{[1 + b(x)|\un|] }
\cdot\bigg(\frac{\nabla\un}{[1 + b(x)|\un|] }-\frac{\nabla T_k(u)}{[1 + b(x)|\un|] }\bigg)\,.
\end{array}
$$
By \rife{weakconv}, assumptions \rife{hypj} and \rife{hypjp}, and Lebesgue's theorem, we have
$$
\arrstre
\begin{array}{r@{\hspace{2pt}}c@{\hspace{2pt}}l}
\disp
\liminf_{n \to +\infty}
\io \frac{j(x,\nabla\un)}{[1 + b(x)|\un|]^{2}}
& \geq &
\disp
\io \frac{j(x,\nabla T_k(u))}{[1 + b(x)|u|]^{2}}
\\
& &
\disp
-
\io \frac{j_{\xi}(x,\nabla T_k(u))}{[1 + b(x)|u|] }
\cdot \frac{\nabla[u-T_k(u)]}{[1 + b(x)|u|]}\,,
\end{array}
$$
that is, since $j_\xi(x,\nabla T_k(u)) \cdot \nabla (u-T_k(u)) = 0$,
$$
\io \frac{j(x,\nabla T_k(u))}{[1 + b(x)|u|]^{2}}
\leq
\liminf_{n \to +\infty}
\io \frac{j(x,\nabla\un)}{[1 + b(x)|\un|]^{2}}\,.
$$
Letting $k$ tend to infinity, and using Levi's theorem, we obtain
 \begin{equation}
\label{levi}
\io \frac{j(x,\nabla u)}{[1 + b(x)|u|]^{2}}
\leq
\liminf_{n \to +\infty}
\io \frac{j(x,\nabla\un)}{[1 + b(x)|\un|]^{2}}\,.
\end{equation}
Inequality (\ref{levi}) and Lemma \ref{lemma_convergenze} imply that
$$
\arrstre
\begin{array}{l}
\disp
\liminf_{n \to +\infty}
\io \frac{j(x,\nabla\un)}{[1 + b(x)|\un|]^{2}}+\frac12\io|\un|^2
-\io f_n\un
\\
\disp
\quad
\geq
\io \frac{j(x,\nabla u)}{[1 + b(x)|u|]^{2}}+\frac12\io|u|^2
-\io fu\,.
\end{array}
$$
Thus, for every $v$ in $\huz$,
$$
\io\! \frac{j(x,\nabla u)} {[1+b(x)|u|]^{2}}
+ \frac12\io\!|u|^2 - \io f u
\leq
\io\! \frac{j(x,\nabla v)} {[1+b(x)|v|]^{2}}+\frac12\io\!|v|^2
-\io f v\,,
$$
so that $u$ is a minimum of $J$; its regularity has been proved in Lemma \ref{lemma_convergenze}.
\end{proof}

\begin{ohss}\label{4maggio}\rm
If we suppose that the coefficient $b(x)$ satisfies the stronger assumption
$$
0<A\leq b(x)\leq B\,,\quad\mbox{almost everywhere in $\Omega$,}
$$ 
it is possible to prove that $J(u) \leq J(w)$ not only for every $w$ in $\huz$, but also for the test functions $w$ such that
\begin{equation}
\label{test-log}
\begin{cases}
 \hbox{$T_k(w)$ belongs to $\huz$ for every $k>0$,}
 \\
\hbox{$\log(1+A|w|)$ belongs to $\huz$,}
\\
\hbox{$w$ belongs to $\elle2$}.
\end{cases}
\end{equation}
Indeed, if $w$ is as in \rife{test-log}, we can use $T_k(w)$
as test function in \rife{enunciato} and we have
$$
J(u)
\leq
J(T_k(w)) =
\io\frac{j(x,\nabla T_k(w))} {[1+b(x)|T_k(w)|]^{2}}+\frac12\io|T_k(w)|^2
-\io f T_k(w).
$$
In the right hand side is possible to pass to the limit, as $k$ tends to infinity, so that we have
$J(u) \leq J(w)$, for every test function $w$ as in \rife{test-log}.
\end{ohss}

\begin{ohss}\label{coerc}\rm
We explicitly point out the differences, concerning the coercivity, between the functionals studied in \cite{BO-Pisa} and the functionals studied in this paper. Indeed, let $0 < \rho < \frac{N-2}{2}$, and consider the sequence of functions
$$
v_{n} = {\rm exp}\left[T_{n}\left( \frac{1}{|x|^{\rho}} - 1\right)\right] - 1\,,
$$
defined in $\Omega = B_{1}(0)$. Then
$$
\log(1 + |v_{n}|) = T_{n}\left( \frac{1}{|x|^{\rho}} - 1\right)\,,
$$
is bounded in $\huz$ (since the function $v(x) = \frac{1}{|x|^{\rho}} - 1$ belongs to $\huz$ by the assumptions on $\rho$), but, by Levi's theorem,
$$
\lim_{n \to +\infty}\,\io |\nabla v_{n}| = \rho \io \frac{{\rm exp}\left[\frac{1}{|x|^{\rho}} - 1\right]}{|x|^{\rho+1}} = +\infty\,.
$$
Hence, the functional
$$
v \in \huz \mapsto \io \frac{|\nabla v|^{2}}{(1 + |v|)^{2}} = \io |\nabla \log(1 + |v|)|^{2}\,,
$$
which is of the type studied in \cite{BO-Pisa}, is non coercive on $\ww$. On the other hand, recalling \rife{coercitivo}, we have
$$
\io |\nabla v| = \io \frac{|\nabla v|}{1 + |v|}\,(1 + |v|) \leq \frac12 \io \frac{|\nabla v|^{2}}{(1 + |v|)^{2}} + \frac12 \io (1 + |v|)^{2}\,.
$$
Thus, the functional
$$
v \in \huz \mapsto \io \frac{|\nabla v|^{2}}{(1 + |v|)^{2}} + \io |v|^{2}\,,
$$
which is of the type studied here, is coercive on $\ww$.
\end{ohss}
 


\begin{thebibliography}{99}

\bibitem{bbl}
R. Benguria, H. Brezis, E.H. Lieb:
{\sl The Thomas-Fermi-von Weizs\"acker theory of atoms and molecules}, 
Comm. Math. Phys. {\bf 79} (1981), 167--180. 

\bibitem{BCO}
L. Boccardo, G. Croce, L. Orsina:
{\sl
A semilinear problem with a $W^{1,1}_0$ solution}, submitted.

\bibitem{BO-Pisa}
L. Boccardo, L. Orsina:
{\sl Existence and regularity of minima
 for integral functionals noncoercive in the energy space.
 Dedicated to Ennio De Giorgi}, Ann. Scuola Norm. Sup. Pisa Cl. Sci. {\bf 25} (1997), 95--130.


\bibitem{DG}
E. De Giorgi:
{\sl Teoremi di semicontinuit\`a nel calcolo delle variazioni}. Lezioni tenute all'Istituto Nazionale di Alta Matematica, Roma, 1968--69; appunti redatti da U. Mosco, G. Troianiello, G. Vergara.


\end{thebibliography}
\end{document}